\documentclass[11pt, amsfonts]{amsart}

\usepackage{amsmath,amssymb,amsthm}
\usepackage[all]{xy}
\usepackage[dvipdfm,colorlinks=true]{hyperref}
\usepackage{mathtools}

\numberwithin{equation}{section}

\newtheorem{thm}{Theorem}[section]
\newtheorem{prop}[thm]{Proposition}
\newtheorem{lem}[thm]{Lemma}

\theoremstyle{definition}

\theoremstyle{rem}

\newcommand{\G}{\mathcal{G} }
\newcommand{\Z}{\mathbb{Z}}
\newcommand{\Q}{\mathbb{Q}}

\newcommand{\map}{\mathrm{Map}}
\newcommand{\ch}{\mathrm{ch}}
\newcommand{\fsus}{\hat{\sigma}^2}
\newcommand{\F}{\mathbb{F}}

\title[Cohomology of the classifying spaces of $U(n)$-gauge groups over $S^2$]{Cohomology of the classifying spaces of $U(n)$-gauge groups over $S^2$}

\author{Masahiro Takeda}
\address{Department of Mathematics, Kyoto University, Kyoto, 606-8502, Japan}
\email{m.takeda@math.kyoto-u.ac.jp}

\begin{document}

\baselineskip.525cm

\maketitle

\begin{abstract}

A gauge group is the topological group of automorphisms of a principal bundle. We compute the integral cohomology ring of the classifying spaces of gauge groups of principal $U(n)$-bundles over the 2-sphere by generalizing the operation for free loop spaces, called the free double suspension.

\end{abstract}


\section{Introduction}

Let $G$ be a topological group, and $P \rightarrow X$ be a principal $G$-bundle over a base $X$.  An automorphism 
of $P$ is, by definition, a $G$-equivariant self-map of $P$ covering the identity map of $X$. The gauge group of $P$, denoted by $\G(P)$, is defined by the topological group of
automorphisms of $P$. 

As in \cite{AB,G}, there is a natural equivalence 
\begin{equation}\label{AB}
B\G(P) \simeq \map (X, BG; \alpha ), \tag{$\star$}
\end{equation}
where $\map(X,Y;f)$ is the path-component of the space of maps $\map(X,Y)$ containing a map $f \colon X \rightarrow Y$ and $\alpha \colon X \rightarrow BG$ is a classifying map of $P$. This connection enables us to employ new techniques and insights from gauge group specifically fiberwise homotopy theory and group theory, to study the homotopy theory of mapping spaces and to import rich tools in the homotopy theory of mapping space to gauge groups. Moreover since the classifying space $B\G(P)$ is homotopy equivalent to the moduli space of connections of $P$ in smooth case, the homotopy theory of gauge groups has potentially application in geometry and physics.

In this paper, we determine the integral cohomology ring of the classifying spaces of gauge groups of principal $U(n)$-bundle over $S^2$. 
Although the (co)homology of the classifying spaces is an obviously important object in topology and has possible applications to geometry and physics, there are only a few previous works. 
Mod-$p$ homology is computed in \cite{C,KT}, and a partial cohomology is done in \cite{K,M,T}. In \cite{AB} the rational Poincar\'e series is determined, and in fact the rational cohomology is trivially determined because the rationalization of $BG$ is a product of Eilenberg-Mac Lane spaces.

Thus our result is the first complete determination of the integral cohomology ring of the classifying space of gauge groups in the nontrivial case.

To state the main theorem we set notation. Let $e\colon \G(P)\rightarrow G$ be a homomorphism given by the evaluation map at the base point of $X$. Then one gets the induced map $B\G(P)\rightarrow BG$ which is denoted by the same symbol $e$. Let $P_{n,k}$ be a principal $U(n)$-bundle over $S^2$ such that $c_1(P_{k,n})=k\in \Z \cong H^2(S^2;\Z)$. Recall that $H^*(BU(n))\cong \Z[c_1,c_2 \dots c_n]$, when $c_i$ is the $i$-th Chern class of the universal bundle. As we will see later, $e \colon B\G(P)\rightarrow BG$ is an injection in cohomology, and so we abbreviate $e^*(c_i)$ by $c_i$. Let $e_i(a_1, a_2, \dots) \in \Z[[a_1, a_2 \dots ]]$ be the $i$-th elementary symmetric function in $a_1, a_2, \dots$, and $s_i$ be the $i$-th Newton polynomial defined by 
$$s_i(e_1,\dots e_i)= \sum_{j}a_j^i,$$ 
where we abbreviate $e_i(a_1,a_2 \dots)$ and $s_i(x_1, \dots x_i)$ by $e_i$ and $s_i$ respectively when the indeterminates are obvious.

Now we state the main theorem.

\begin{thm}\label{main}
There is an isomorphism 
$$H^*(B\G(P_{n,k});\Z) \cong \Z[c_1, \dots c_n, x_1,x_2, \dots]/(h_n, h_{n+1},\dots)$$
such that there is $\zeta \in K(B\G(P_{n,k}))$ with $c_i(\zeta) = x_i$, where $$h_i= kc_i + \sum_{1 \leq j \leq i} (-1)^{j}s_{j}(x_1, x_2 \dots x_{j})c_{i-j}.$$
\end{thm}

To prove this theorem we will generalize a certain map in the cohomology of free loop spaces which is defined in \cite{KK} and called the free loop suspension. Since the homotopy equivalence (\refeq{AB}) is natural with respects to $X$, the map $e\colon B\G(P) \rightarrow BG$ coincides with the evaluation map $\map (X, BG; \alpha )\rightarrow BG$ at the base point of $X$  which is ambiguously denoted by the same symbol $e$. Specifically in the case of $P_{n,k}$ there is a evaluation fibration 
$$\Omega^2_kBU(n) \rightarrow \map(S^2, BU(n);k) \xrightarrow{e} BU(n),$$ 
where $\Omega^2_kBU(n)$ and $\map(S^2, BU(n);k)$ are the connected component of the double loop space of $BU(n)$ and $\map(S^2,BU(n))$ containing a dgree $k$ map respectively. 
Since $\Omega^2_kBU(n) \simeq \Omega_0 U(n)$, $\Omega^2_kBU(n)$ and $BU(n)$ have only even cells. Thus associated Serre spectral sequence of this evaluation fibration collapses at the $E_2$-term, hence there is an isomorphism as $H^*(BU(n))$-modules
$$H^*(\map(S^2, BU(n);k)) \cong H^*(\Omega^2_kBU(n)) \otimes H^*(BU(n)),$$
and so in particular $e$ is an injection in cohomology. Thus it remains to determine the ring structure by using the free double suspnsion.

\section{Free double suspension}

Let $LX \coloneqq \map(S^1,X)$, 
the free loop space of a space $X$. In \cite{KK}, a map $\hat{\sigma}\colon H^*(X)\to H^{*-1}(LX)$ is constructed as an extension of the cohomology suspension $\sigma\colon H^*(X)\to H^{*-1}(\Omega X)$ and apply it to the evaluation fibration
$$\Omega X\to LX\to X$$
to determine the cohomology of $LX$. In this section, we generalize the free suspension to a mapping space $\map(S^2,X;f)$ and show its basic properties that we are going to use. Since in \cite{KK} the space $X$ is assumed to be simply connected, the components of $LX$ is not necessary considered in this case. We have to consider the path-components of  $\map(S^2,X)$ in general free suspension.

Let $\hat{e}^2\colon S^2\times\map(S^2,X;f)\to X$ be the evaluation map defined by $\hat{e}(s,g)=g(s)$ for $(s,g)\in S^2\times\map(S^2,X)$. We define the \emph{free double suspension}
$$\hat{\sigma}^2_f\colon H^*(X)\to H^{*-2}(\map(S^2,X;f))\quad\text{by}\quad\hat{\sigma}^2_f(x)=\hat{e}^{*}(x)\slash v$$
for $x\in H^*(X)$, where $\slash$ means the slant product and $v \in H_2(S^2)$ is the Hurewicz image of the identity map of $S^2$.

To state properties of free double suspensions, we set notation. 
Let $e\colon\map(S^2,X;f)\to X$ be the evaluation map at the basepoint of $S^2$.  
Let $\sigma^2_f \colon H^*(X)\to H^{*-2}(\Omega^2_f X)$ be the composite of cohomology suspensions and the inclusion map
$$H^*(X)\xrightarrow{\sigma}H^{*-1}(\Omega X)\xrightarrow{\sigma}H^{*-2}(\Omega^2 X)\rightarrow H^{*-2}(\Omega^2 _fX),$$
where $f\in \Omega^2X$.
Let $j\colon\Omega^2_f X\to \map(S^2,X;f)$ be the inclusion.

\begin{prop}\label{basic}
  Let $f \in \Omega^2(X)$. Free double suspensions have the following properties.
  \begin{enumerate}
    \item  $\hat{\sigma}^2_f$ restricts to $\sigma^2_f$ such that
    $$j^*\circ\hat{\sigma}^2_f=\sigma^2_f.$$
    \item $\hat{\sigma}^2_f$ is a derivation such that for $x,y\in H^*(X)$
    $$\hat{\sigma}^2_f(xy)=\hat{\sigma}^2_f(x)e^*(y)+e^*(x)\fsus_f(y).$$
    \item Suppose that $X$ is a path-connected H-group with a multiplication $\mu\colon X\times X\to X$. If $\mu^*(y)=\sum_iy_i\times y_i'$ for $y\in H^*(X)$, then
    $$\hat{\sigma}^2_f(y)=\sum_i\alpha^*(\sigma^2_f(y_i))e^*(y_i'),$$
    where $\alpha\colon\map(S^2,X;f)\to\Omega^2_fX$ is given by $\alpha(g)=g\cdot e(g)^{-1}$ for $g\in\map(S^2,X;f)$.
  \end{enumerate}
\end{prop}

\begin{proof}
  (1) Let $\bar{e}\colon S^2\times\Omega^2_f X\to X$ be the restriction of $\hat{e}$, that is, $\bar{e}=\hat{e}\circ(1\times j)$. Recall that $\sigma^2_f(x)=\bar{e}^*(x)\slash v$ for $x\in H^*(X)$. Then
  $$j^*\circ\hat{\sigma}^2_f(x)=j^*(\hat{e}^*(x)\slash v)=((1\times j)^*\circ\hat{e}^*(x))\slash v=\bar{e}^*(x)\slash v=\sigma^2_f(x).$$

  (2) By definition, $\hat{e}^*(x)=1\times e^*(x)+u\times\hat{\sigma}^2_f(x)$ for $x\in H^*(X)$, where $u \in H^2(S^2)$ is the Kronecker dual of $v$. Then for $x,y\in H^*(X)$,
  \begin{align*}
    \hat{e}^*(xy)&=\hat{e}^*(x)\hat{e}^*(y)=(1\times e^*(x)+u\times\hat{\sigma}^2_f(x))(1\times e^*(y)+u\times\hat{\sigma}^2_f(y))\\
    &=1\times e^*(xy)+u\times(\hat{\sigma}^2_f(x)e^*(y)+e^*(x)\hat{\sigma}^2_f(y)).
  \end{align*}
  Thus one gets the desired equality by taking the slant product with $v$.

   (3) The map $\alpha\times e\colon\map(S^2,X;f)\to\Omega^2_f X\times X$ is obviously a homotopy equivalence and satisfies a homotopy commutative diagram
  $$\xymatrix{S^2\times\map(S^2,X;f)\ar[rr]^{1\times(\alpha\times e)}\ar[d]^{\hat{e}}&&S^2\times\Omega^2_fX\times X\ar[d]^{\bar{e}\times 1}\\
  X&&X\times X\ar[ll]_\mu.}$$
  As well as $\hat{\sigma}_f$, one has $\bar{e}^*(y)=1\times j^* \circ e^*(y)+u\times\sigma^2_f(y)$ for $y\in H^*(X)$. Thus if $\mu^*(y)=\sum_iy_i\otimes y_i'$, then
  \begin{align*}
    \hat{e}^*(x)&=(1\times(\alpha \times e)^*)\circ(\bar{e}^*\times 1)\circ\mu^*(y)=(1\times\alpha^*\times e^*)\circ(\bar{e}^*\times 1)\Bigl(\sum_iy_i\times y_i'\Bigr)\\
    &=(1\times (\alpha\times e)^*)\Bigl(\sum_i(1\times j^* \circ e^*(y_i)+u\times\sigma^2_f(y_i))\times y_i'\Bigr)\\
    &=\sum_i(1\times \alpha^* \circ j^* \circ e^*(y_i)+u\times\alpha^*(\sigma^2_f(y_i))) e^*(y_i').
  \end{align*}
  Therefore the desired equality is obtained by taking the slant product with $v$.
  
\end{proof}

We observe the relation between cohomology suspensions and component shifts. Let $\bar{f} \colon \Omega^2_0X \rightarrow \Omega^2_fX$ denote the map given by adding a map $f \colon S^2 \rightarrow X$. Then $\bar{f}$ is a homotopy equivalence.

\begin{lem}\label{shift}
For a map $f\colon S^2 \rightarrow X$ and $x\in H^*(X)$,
$$\sigma^2_f(x) =(f^*(x) \times 1) /v+(\bar{f}^*)^{-1}\circ \sigma^2_0(x).$$

\end{lem}

\begin{proof}
Let $\phi \colon (S^2 \vee S^2) \times \Omega^2_0 X \rightarrow X $ be the map that $\phi (s,*,g)=f (s)$ and $\phi (*,t,g)=g(t)$, and $\eta\colon S^2 \rightarrow  S^2 \vee S^2$ be the folding map. Then there is a homotopy commutative diagram

$$\xymatrix{
  S^2\times\Omega^2_f X\ar[d]^{\bar{e}} \ar[rr]^{1 \times \bar{f}^{-1}} && S^2 \times \Omega^2_0X \ar[d]^{\eta \times 1 }  \\
  X && (S^2\vee S^2)\times \Omega^2_0 X \ar[ll]^{\phi}.
}$$

Thus 
\begin{align*}
\bar{e}^* (x)&=(1 \times (\bar{f}^*)^{-1})\circ \eta \circ \phi (x)\\
&=(1 \times (\bar{f}^*)^{-1})\circ \eta((f^*(x)\vee 1)\times 1 +(1 \vee u)\times \sigma^2_0(x)+1 \times j^*\circ e^*(x))\\
&=f^{*}(x)\times 1 +u \times (\bar{f}^*)^{-1}(\sigma^2_0(x)) +1 \times  (\bar{f}^*)^{-1}\circ j^*\circ e^*(x).
\end{align*}
The desired equality is obtained by taking the slant product with $v$.

\end{proof}

\section{Proof of the main theorem}

In this section, we prove the main theorem. 
Let $i_n \colon SU(n) \rightarrow SU(\infty)$ be the inclusion.

\begin{prop}\label{Bott}

$\quad$

\begin{enumerate}

\item There is an isomorphism $$H^*(\Omega SU(n)) \cong \Z[y_1, y_2, \dots ]/(s_n, s_{n+1}, \dots), \quad |y_i| = 2i.$$
\item The map $\Omega i_n \colon \Omega SU(n) \rightarrow \Omega SU(\infty)$ is a surjection in cohomology.

\end{enumerate}

\end{prop}

\begin{proof}
$\quad$

\begin{enumerate}
\item This follows from the result of Bott \cite[Proposition 8.1]{Bo}.
\item 
By the construction of Bott \cite{Bo}, the isomorphism of (1) is natural with respect to the inclusion $i_n \colon \Omega SU(n) \rightarrow \Omega SU(n+1)$. Namely, $(i_n)^*(y_i)=y_i$ for each $i$.
\end{enumerate}

\end{proof}

We set notation. Let $\iota \in\widetilde{K}(S^2)$ be a generator such that $\ch(\iota)=u$, where $\ch\colon K(X)\to H^{**}(X;\Q)$ denotes the Chern character and $u$ is as in Section2.
Let $\xi_n \in \widetilde{K}(BU(n))$ be the universal bundle over $BU(n)$ minus the rank $n$ trivial bundle, and $\xi_{\infty} \coloneqq \text{colim}\xi_n$. We define $\beta \colon BU(\infty) \rightarrow \Omega^2_0BU(\infty)$ be the adjoint of $\iota \wedge \xi_\infty \colon S^2 \wedge BU(\infty) \rightarrow BU(\infty)$. 
Let $j_n \colon \map(S^2,BU(n);k) \rightarrow \map(S^2,BU(\infty);k)$ be the map induced by the inclusion $i_n\colon BU(n) \rightarrow BU(\infty)$.

\begin{lem}\label{sigma}

In $H^*(\Omega^2_0BU(\infty))$, 

\begin{equation*}
\sigma^2_0(c_m) = \begin{cases}	
0 & (m=1)\\
(-1)^{m-1} (\beta^{*})^{-1}(s_{m-1}(c_1, \dots c_{m-1})) & (m\geq 2).
\end{cases}
\end{equation*}

\end{lem}

\begin{proof}
There is  a homotopy commutative diagram
$$
\xymatrix{
S^2 \wedge BU(\infty) \ar[r]^{\quad \iota \wedge \xi_\infty }\ar[d]_{1 \times \beta} & BU(\infty) \\ 
S^2 \wedge \Omega^2_0BU(\infty) \ar[ru]_{\bar{e}} & \\
}
$$
where $\bar{e}$ is as in Section2.
Then it follows that
\begin{align*}
u\times \ch(\xi_\infty)&=\ch(\iota) \times \ch(\xi_\infty) =(\iota \wedge \xi_\infty)^*(\ch(\xi_\infty))\\
&=(1 \times \beta)^* \circ \bar{e}^*(\ch(\xi_\infty))= u \times \beta^*(\sigma^2_0(ch(\xi_\infty))).
\end{align*}
in the rational cohomology. Thus since $\ch(\xi_\infty)=\sum_{i\geq1}\frac{1}{i!}s_i$,
$$u\times s_m = u \times \beta^{*}(\sigma^2_0((-1)^{m}c_{m+1}))$$  in the integral cohomology. Completing the proof.

\end{proof}

Let $\gamma_k \colon \Omega^2_0BU(n) \rightarrow \Omega^2_kBU(n)$ be the map given by the concatenation with the degree $k$ map $S^2 \rightarrow BU(n)$. Then $\gamma_k$ is a homotopy equivalence.

\noindent\textit{Proof of Theorem \ref{main}}

First, we define $$\bar{x}_m \coloneqq \alpha^* \circ {({\gamma_k}^*)}^{-1} \circ  {(\beta^*)}^{-1} (c_m) \quad\text{and}\quad x_m \coloneqq j_n^* (\bar{x}_m),$$ where the map $\alpha \colon \map(S^2,BU(n);k) \rightarrow \Omega^2_kBU(n)$ is as in Section 2. We show these $x_i$ become the Chern classes of a virtual bundle in the latter half of this proof.

There is a homotopy commutative diagram 
$$
\xymatrix{
\Omega^2_0BU(n) \ar[r]^{\Omega^2_0i_n} \ar[d]^{\gamma_k} & \Omega^2_0BU(\infty) \ar[d]^{\gamma_k} \\ 
\Omega^2_kBU(n) \ar[r]^{\Omega^2_ki_n} & \Omega^2_kBU(\infty) 
}
$$
in which vertex maps are homotopy equivalence.
Then by Lemma \ref{Bott} (2), $\Omega^2_k i_n$ is surjective in cohomology.
Thus since $H^{*}(\Omega^2_kBU(\infty))$ is generated by $ ({\gamma^*_k})^{-1} \circ  {(\beta^*)}^{-1} (c_i)$ for $i\geq 1$, $H^{*}(\Omega^2_kBU(n)) $ is generated by  $(\Omega_k^2i_n)^* \circ ({\gamma^*_k})^{-1} \circ  {(\beta^*)}^{-1} (c_i)$ for $i\geq 1$. 
There is a homotopy commutative diagram

$$
\xymatrix{
\Omega^2_kBU(n) \ar[r]^{\Omega^2_k i_n} \ar[d]   & \Omega^2_kBU(\infty) \ar@{=}[r]\ar[d] & \Omega^2_kBU(\infty)  \ar@{=}[d]  \\ 
\map (S^2, BU(n);k) \ar[r]^{j_n}&\map(S^2, BU(\infty);k)  \ar[r]^{\quad \alpha}  &\Omega^2_kBU(\infty).
}
$$
Then $x_i$ restricts to $(\Omega_k^2i_n)^* \circ ({\gamma^*_k})^{-1} \circ  {(\beta^*)}^{-1} (c_i)$.
Now we apply the Leray-Hirsch theorem to the evaluation fibration $\Omega^2_k(BU(n)) \xrightarrow{j} \map(S^2, BU(n);k) \xrightarrow{e} BU(n),$ we obtain that
$$\Phi \colon \Z[c_1, c_2, \dots c_n, x_1, x_2, \dots] \rightarrow H^*(\map(S^2, BU(n);k))$$
is surjective.

We next show $h_i \in \text{Ker}(\Phi)$. By Proposition \ref{basic} (3), Lemma \ref{shift} and Lemma \ref{sigma}, in $H^*(\map(S^2,BU(\infty);k))$ 
\begin{align*}
\fsus_k(c_i) & = k c_{i-1} + \sum_{2\leq j \leq i}\alpha^* \circ {(\gamma_k^*)}^{-1}(\sigma^2_0(c_j))c_{i-j} \\
& = kc_{i-1} + \sum_{2\leq j \leq i}(-1)^{j-1} s_{j-1}(\bar{x}_1, \bar{x}_2, \dots) c_{i-j}.
\end{align*}
Then for $i\geq n$
$$\Phi(h_i)=j_n^{*}(\fsus_k(c_{i+1})) =\fsus_k(i_n^{*}(c_{i+1}))=0$$
where $j_n^*(\bar{x}_i)=x_i$.
Thus $\Phi$ induces a surjection
$$\bar{\Phi} \colon \Z[c_1, c_2, \dots c_n, x_1, x_2, \dots]/(h_n, h_{n+1}, \dots) \rightarrow H^*(\map(S^2, BU(n);k)).$$

We next show that $\bar{\Phi}$ is an isomorphism.
Let $\F$ be an arbitrary field. We calculate the Poincar\'e series of $A_{\F} \coloneqq \F[c_1, c_2, \dots c_n, x_1, x_2, \dots]/(h_n, h_{n+1}, \dots).$ 
Let $P_t(V)$ be the Poincar\'e series of a graded vector space $V$.
Since $h_i \equiv s_i(x_1, x_2, \dots x_i) \quad \text{mod} (c_1, c_2, \dots c_n)$ and $c_1, c_2, \dots c_n$ is a regular sequence in $A_\F$,
$$P_t(\F[c_1, c_2, \dots c_n, x_1, x_2, \dots]/(h_n, h_{n+1}, \dots))=\frac{P_t(\F[x_1, x_2, \dots]/(s_n, s_{n+1} \dots))}{(1-t^2)(1-t^4)\dots(1-t^{2n})}.$$
Then by Lemma \ref{Bott},
$$P_t(A_\F)=\frac{P_t(H^{*}(\Omega SU(n);\F))}{(1-t^2)(1-t^4)\dots(1-t^{2n})}.$$
On the other hand as in Section 1 the Serre spectral sequence of the evaluation fibration $\Omega^2_k(BU(n)) \rightarrow \map(S^2, BU(n);k) \xrightarrow{e} BU(n)$ collapses at the $E_2$-term, and so 
\begin{align*}
P_t(H^*(\map(S^2, BU(n);k);\F))&=P_t(H^{*}(BU(n);\F))\times P_t(H^{*}(\Omega SU(n);\F))\\
&=\frac{ P_t(H^{*}(\Omega SU(n);\F))}{(1-t^2)(1-t^4)\dots(1-t^{2n})}.
\end{align*}
Then we get the equality $$P_t(A_\F)=P_t(H^*(\map(S^2, BU(n);k);\F)).$$
Since the source and target of $\bar{\Phi}$ is of finite type, $\bar{\Phi}$ is an isomorphism over an arbitrary field.
Thus $\bar{\Phi}$ is an isomorphism over $\Z$.


 It remains to show that $x_i$ can be represented as the Chern class of a vertual bundle. 
Since the K\"{u}nneth formula holds as
$$K(S^2\times \map(S^2,BU(\infty);k))\cong K(S^2)\otimes K(\map(S^2,BU(\infty);k)),$$
we can define the $K$-theoretic free double suspension $\fsus_k$. 
We define
$$\hat{\sigma}^2_{f}\colon K(BU(\infty))\to K(\map(S^2,BU(\infty);f))\quad\text{by}\quad\hat{e}^*(x)=1\otimes e^{*}(x)+\iota\otimes\hat{\sigma}^2_{f}(x)$$
for $x\in K(BU(\infty))$. 
By the same argument as in the first half of the proof of this theorem, 
\begin{align*}
\fsus_{k}(\xi_\infty)=k +\alpha^* \circ (\gamma_k^{*})^{-1}\circ(\beta^*)^{-1}(\xi_\infty).
\end{align*}
If we put $\zeta_\infty\coloneqq \fsus_{k}(\xi_\infty)$, then $c_i(\zeta_\infty)=\bar{x}_i$ for $i\geq 1$. Let $\zeta_n \coloneqq j_n^*(\zeta_\infty)$
then $c_i(\zeta_n)=j_n^*(c_i(\zeta_\infty))=x_i$ as desired. 
Therefore the proof is complete.

\begin{flushright}
$\square$
\end{flushright}

\section*{Acknowledgement}

The author is deeply grateful to Daisuke Kishimoto for much valuable advice.

\end{document}